\documentclass[reqno]{amsart}
\usepackage{fullpage}
\usepackage{array}
\usepackage{hyperref}
\usepackage{graphicx,color}
\usepackage{float}
\usepackage{upgreek}

\newcommand{\red}{\color{red}}

\newcommand{\gen}[1]{\ensuremath{\langle{#1}\rangle}}
\newcommand{\N}{\ensuremath{\mathbb{N}}}
\newcommand{\Z}{\ensuremath{\mathbb{Z}}}

\newcommand{\ds}{\displaystyle}

\newcommand{\id}{\ensuremath{\textnormal{Id}}}

\newcommand{\Id}{\textnormal{Id}}
\newcommand{\Idg}{\textnormal{Id}^{gr}}

\newcommand{\spn}{\textnormal{span}}

\newcommand{\F}{F\langle X\rangle}

\newcolumntype{C}[1]{>{\centering\let\newline\\\arraybackslash\hspace{0pt}}m{#1}}

\usepackage[utf8]{inputenc}
\usepackage{amsfonts}
\usepackage{amssymb}
\usepackage{verbatim}
\usepackage{mathrsfs}

\theoremstyle{definition} 

\newtheorem{teo}{Theorem}[section]
\newtheorem{cor}[teo]{Corollary}
\newtheorem{lema}[teo]{Lemma}

\newtheorem{pr}[teo]{Proposition}

\newtheorem{obs}[teo]{Remark}

\numberwithin{equation}{section}

\begin{document}
	\title{Central cocharacters of the subvarieties of varieties \\ of superalgebras with almost polynomial growth}

\author[W. D. S. Costa, J. P. Cruz, T.S. do Nascimento and A. C. Vieira]{W. D. S. Costa, J. P. Cruz$^{1}$,  T. S. do Nascimento and A. C. Vieira$^{1, 2, *}$}
	
\dedicatory{Departamento de Matemática, Instituto de Ciências Exatas, Universidade Federal de Minas Gerais. \\ Avenida Antonio Carlos 6627, 31123-970, Belo Horizonte, Brazil.\\
Departamento de Matemática Pura e Aplicada, Centro de Ciências Exatas e da Saúde, Universidade Federal do\\ Espírito Santo.
Alto Universitário, s/n, 29500-000, Alegre, Brazil\\
Departamento de Matemática, Instituto de Ciências Exatas e da Terra, Universidade Federal de Mato Grosso. \\ Avenida Fernando Correa da Costa 2367, 78060-900, Cuiabá, Brazil.}	
	
	\thanks{\footnotesize $^{1}$ Partially supported by FAPEMIG}
	\thanks{\footnotesize $^{2}$ Partially supported by CNPq}
	\thanks{\footnotesize {$^*$ Corresponding author}}
	\thanks{\footnotesize {\it E-mail addresses}:  willer.costa@ufes.com (Costa), juan.mat10@gmail.com (Cruz), thais.nascimento@ufmt.br (Nascimento),  anacris@ufmg.br (Vieira)}

	\subjclass[2020]{Primary 16R10, 16R50, Secondary 16W50, 20C30}
	
	\keywords{polynomial growth, central codimensions, graded algebras}
	
	\begin{abstract}  
In recent years, the study of the $T$-space of central polynomials of an algebra $A$ has become an object of great interest in the PI-theory. Such interest has been extended to the context of algebras with additional structures. The main goal of this paper is to present information about the central graded codimensions and the central graded cocharacters of the varieties of superalgebras  $\mathrm{var}^{gr}(\mathcal{G})$, $\mathrm{var}^{gr}(UT_2)$, 
$\mathrm{var}^{gr}(\mathcal{G}^{gr})$, $\mathrm{var}^{gr}(UT^{gr}_2)$ and $\mathrm{var}^{gr}(D^{gr})$, which are the only supervarieties with almost polynomial growth of the graded codimensions. Also we establish the generators of the space of central polynomials, determine the central codimensions and explicitly give the decomposition of the central graded cocharacters of each minimal subvariety of such supervarieties.
	\end{abstract}
 	\maketitle

\section{Introdution}

The sequences of codimensions and central codimensions of an algebra over a field of characteristic zero have been an important object of study in the last few years (see \cite{BrKoKrSi}, \cite{GZcentasso2}, \cite{GZcentasso}, \cite{GuiFidKo}, \cite{LaMaMarRi}, \cite{reggrow}). Recall that for a field of characteristic zero  $F$, we may consider the free associative algebra $\F$ over $F$ on a countable set $X = \{x_{1}, x_{2}, \ldots\}$. Given an algebra $A$ over $F$, we denote its center by $Z(A)$ and say that a polynomial $f(x_{1},\ldots,x_{n}) \in \F$ is a central polynomial of $A$ if it has zero constant term and $f(a_{1}, \ldots, a_{n}) \in Z(A)$, for all $a_{1}, \ldots, a_{n} \in A$. When  $f$ vanishes under all evaluations by elements in $A$,  $f$  is called an identity of $A$, otherwise $f$ is a proper central polynomial.

Regev in \cite{regexis,reggrow} introduced the concepts of codimension and central codimensions of an algebra $A$ as follows. Let $P_{n}$ the space of the multilinear polynomials in the variables $x_{1}, \ldots, x_{n}$ and consider the quotient spaces
$$P_{n}(A) = \displaystyle\frac{P_{n}}{P_{n} \cap \Id(A)} \ , \ P_{n}^{z}(A) = \displaystyle\frac{P_{n}}{P_{n} \cap C(A)} \ \ \mbox{and} \ \ \Delta_{n}(A) = \displaystyle\frac{P_{n} \cap C(A)}{P_{n} \cap \Id(A)},$$
where $C(A)$ is the space of the central polynomials and  $\Id(A)$ is the ideal of identities of $A$.

For $n \geq 1$, we define $n$th codimension, $n$th central codimension and the $n$th proper central codimension of $A$ by $c_{n}(A) = \dim P_{n}(A)$, $c_{n}^{z}(A) = \dim P_{n}^{z}(A)$ and $\delta_{n}(A) = \dim \Delta_{n}(A)$, respectively. 

Regev \cite{regexis} also proved that the sequence $\{c_n(A)\}_{n\geq 1}$ is exponentially bounded in case $A$ satisfies a non trivial identity.
On the other hand, the algebra $A$ has polynomial growth of the codimensions if there exist constants $a$, $b \geq 0$ such that $c_n(A) \leq an^b$, for all $n\geq 1$ and the first characterization of varieties of algebras with polynomial growth was presented by Kemer in \cite{kemer1978t} as follows.

Consider $\mathcal{G} = \langle 1, e_1, e_2, \ldots \mid e_ie_j = - e_je_i  \rangle$ be the infinite dimensional Grassmann algebra and let $UT_2(F)$ denote the algebra of $2 \times 2$ upper triangular matrices over $F$.  If  $\mathcal{V}= \textnormal{var}(A)$ is the variety of algebras generated by $A$ then the famous Kemer's theorem establishes that $c_n(\mathcal{V}), n = 1,2,\ldots$, is polynomially bounded if and only if 
$\mathcal{G}, UT_2 \notin \mathcal{V}$. Hence $\mathrm{var}(\mathcal{\mathcal{G}})$ and $\mathrm{var}(UT_2)$ are the only varieties of almost polynomial growth, 
i.e., the codimension sequences of those varieties grow exponentially but the codimension sequence of any proper subvariety of them is polynomially bounded.

 In contrast to $\mathcal{G}$, the algebra $UT_2$ does not contain proper central polynomials, and so $c_n^z(UT_2)= c_n(UT_2)$ for all $n\geq 1$. It is worth mentioning that in \cite{reggrow}, Regev determined the central codimension of $\mathcal{G}$, proving that $c_n(\mathcal{G})= 2c^z_n(\mathcal{G})$.  

The theory of varieties developed by Kemer (see \cite{Kem}) shows that the superalgebras and their graded 
identities play an important role in the development of the PI-theory. Given a  superalgebra $A$, we write $\mathcal{V} = \mathrm{var}^{gr}(A)$ for the variety of superalgebras generated by $A$ and we denote by $c^{gr}_n(A), n = 1,2,\ldots$, the sequence of graded codimensions of 
$A$. Recall that $c^{gr}_n(A)$ is the dimension of the space of multilinear polynomials 
in $n$ graded variables in the corresponding relatively free superalgebra of countable rank.

For the infinite dimensional Grassmann algebra,
we write $\mathcal{G}$ when it is a superalgebra  with the trivial grading and $\mathcal{G}^{gr}$ to denote $\mathcal{G}$ with the grading 
$(\mathcal{G}^{(0)}, \mathcal{G}^{(1)})$, where $\mathcal{G}^{(0)}$ is the span of all monomials in the $e_i$'s of even length and 
$\mathcal{G}^{(1)}$ is the span of all monomials of odd length.

Also, let $UT_2$ denote the algebra $UT_2$  as a superalgebra with trivial grading 
and let $UT^{gr}_2$ denote the algebra $UT_2$ with grading 
$
UT^{(0)}_2 = F e_{11} + F e_{22}$ and $UT^{(1)}_2 = F e_{12}.
$
Finally, let $D^{gr}$ be the commutative algebra  $F \oplus F$ with grading $(F(1,1),F(1,-1))$. 

In \cite{GiaMiZai} the authors showed that the above 
five superalgebras characterize the varieties of superalgebras of polynomial growth. In fact, they proved that if $\mathcal{V}$ is a variety of superalgebras then
$
c^{gr}_n(\mathcal{V}^{gr}) \leq k n^t$, for some constants $k, t$
if and only if $\mathcal{G}, \mathcal{G}^{gr}, UT_2, UT^{gr}_2, D^{gr} \notin \mathcal{V}^{gr}$. As a consequence $\mathrm{var}^{gr}(\mathcal{G}),$ $\mathrm{var}^{gr}(UT_2),$ 
$\mathrm{var}^{gr}(\mathcal{G}^{gr}), \mathrm{var}^{gr}(UT^{gr}_2)$ and $ \mathrm{var}^{gr}(D^{gr})$ are the only varieties of superalgebras 
of almost polynomial growth. 

In \cite{LaMat,LaMatalmostGgr} the author presented a complete list 
of superalgebras generating the proper subvarieties of  $\mathrm{var}^{gr}(\mathcal{G})$, $\mathrm{var}^{gr}(UT_2)$, 
$\mathrm{var}^{gr}(\mathcal{G}^{gr})$, $\mathrm{var}^{gr}(UT^{gr}_2)$ and $\mathrm{var}^{gr}(D^{gr})$.
Moreover, among them, they classify the minimal ones, which are proper subvarieties $\mathcal{U}$ such that
$
c^{gr}_n(\mathcal{U}) \approx q n^k \ \text{for some } k \geq 1, q > 0,
$
and for any proper subvariety $\mathcal{W} \subset \mathcal{U}$, 
we have 
$c^{gr}_n(\mathcal{W}) \approx q' n^t \ \text{with } t < k.$

We emphasize that the ideals of graded identities and the graded codimensions of all minimal subvarieties of the varieties of superalgebras with almost polynomial growth were given in  \cite{LaMat} and \cite{LaMatalmostGgr}. Furthermore, in \cite{NaSaVi}, the authors exhibit the decomposition of the sequences of graded cocharacters of all such subvarieties.

The concepts of central polynomials and central codimensions can be extended to the set of superalgebras by considering the corresponding objects.

The main goal of this paper is to present information about the space of the central graded polynomials,  the central graded codimensions and the central cocharacters of the superalgebras of almost polynomial growth and also, of their minimal subvarieties.

To this end, we determine the generators of the space of the central graded polynomials and calculate the central graded codimensions of each minimal subvariety of $\mathrm{var}^{gr}(\mathcal{G})$, $\mathrm{var}^{gr}(UT_2)$, 
$\mathrm{var}^{gr}(\mathcal{G}^{gr})$, $\mathrm{var}^{gr}(UT^{gr}_2)$ and $\mathrm{var}^{gr}(D^{gr})$. Moreover, we explicitly give the decomposition of the central graded cocharacters of such supervarieties.

\section{Generalities}

An algebra $A$ is a superalgebra (or a $\mathbb{Z}_2$-graded algebra) if it admits a decomposition $A=A_{0} \oplus A_{1}$, where $A_{0}$ and $A_{1}$ are subspaces of $A$ satsifying $A_{0}A_{0}+A_{1}A_{1}\subseteq A_{0}$ and $A_{0}A_{1}+A_{1}A_{0}\subseteq A_{1}$. The elements of $A_{0}$ and $A_{1}$ are called homogeneous of degree zero (even elements) and of degree one (odd elements), respectively. The pair $(A_{0}, A_{1})$ is called a grading of $A$. Clearly, any algebra $A$ can be viewed as a superalgebra with grading $(A, \{0\}),$ called the trivial grading of $A$. For this reason, the theory of graded identities generalizes the ordinary theory of polynomial identities.

 A subalgebra (ideal) $B \subseteq A$ is a graded subalgebra (ideal) if $B = (B \cap A_{0}) \oplus (B \cap A_{1})$. It is well known that the Jacobson radical $J(A)$ of an algebra $A$ is a graded ideal of $A$.
 
 The free associative algebra $\F$ has a natural superalgebra structure. Let $X = Y \cup Z$ be the disjoint union of two countable sets. Denote by $F_{0}$ the subspace of $F \langle Y \cup Z \rangle = F \langle Y,Z \rangle$ spanned by all monomials in the variables of $X$ having even degree in the variables of $Z$ and by $F_{1}$ the subspace spanned by all monomials of odd degree in $Z$, then $F \langle Y,Z \rangle = F_{0} \oplus F_{1}$ is a  $\Z_{2}$-graded algebra called the free superalgebra on $Y$ and $Z$ over $F$. 

 Given a superalgebra $A$, a $\Z_{2}$-polynomial $f (y_{1}, \ldots , y_{n}, z_{1}, \ldots, z_{m}) \in F \langle Y,Z \rangle$ is a graded identity of $A$ if
 $$f(a_{1}, \ldots ,a_{n}, b_{1}, \ldots , b_{m}) = 0 \ \mbox{for all} \ \ a_{1}, \ldots, a_{n} \in A_{0}, \ b_{1} \ldots , b_{m} \in A_{1}.$$
 The set of all graded identities of $A$ is denoted by $\Idg(A)$. This set is a $T_{2}$-ideal of $F \langle Y,Z \rangle$, that is, an ideal invariant under all endomorphisms of $F \langle Y,Z \rangle$ that preserve the grading. 
 
 A $\Z_{2}$-polynomial $f(y_{1},\ldots, y_{r},z_{1}, \ldots, z_{s}) \in F \langle Y,Z \rangle$ is a central graded polynomial of $A$ if it has zero constant term and its evaluations yield elements in the center: 
 $$f(a_{1}, \ldots, a_{r},b_{1}, \ldots, b_{s}) \in Z(A) \ \mbox{for all} \ a_{1}, \ldots , a_{s} \in A_{0}, \ b_{1}, \ldots, b_{s} \in A_{1}.$$  
 Clearly, if $f$ takes only the zero value, $f$ is a graded identiy of $A$. If there exist substitutions for which $f$ takes a non-zero value in $Z(A)$, then $f$ is called a proper central graded polynomials. The set of all central graded polynomial of $A$ is denoted by $C^{gr}(A)$. This set is a $T_{2}$-space of $F \langle Y,Z \rangle$, which means that it is a vector space invariant under all endomorphisms of the free algebra that preserve the grading.
 
 For a subset $S\subseteq F\langle Y,Z \rangle$, we denote $\langle S\rangle_{T_{2}}$  the $T_2$-ideal generated by $S$ and define the $T_{2}$-space generated by $S$, denoted by $\langle S\rangle^{T_{2}}$, as the space generated by all polynomials of the form $f(g_1,\dots, g_{2n})$, where $f(y_1,z_{1}, \dots,y_{n},z_{n})\in S$ and $g_1,\dots, g_{2n}\in F\langle Y,Z \rangle$.
 
 Furthermore, for a $T_{2}$-ideal $I\subseteq F\gen{Y,Z}$ and polynomials $f_1,\dots,f_k\in F\gen{Y,Z}$, the $T_{2}$-space generated by
 $\{f_1,\ldots,f_k \}\cup I$ will be denoted by $\langle  f_1,\dots,f_k,~I\, \rangle^{T_{2}} $.

\begin{obs}\label{rie} Let $A$ and $B$ be superalgebras. If $\Idg(A) = \Idg(B)$, then $C^{gr}(A) = C^{gr}(B)$.    
\end{obs}

In contrast to $\Idg(A)$, the set $C^{gr}(A)$ is not necessarily a $T_{2}$-ideal but rather a $T_{2}$-space. A natural question is to determine conditions under which $C^{gr}(A)$ becomes a $T_{2}$-ideal. In the following result, we explore a condition on the center of a unitary superalgebra that ensures its $T_{2}$-space of the central graded polynomials is, in fact, a 
$T_{2}$-ideal.

\

\begin{pr}\label{prop2} Let $A$ be a unitary superalgebra such that $A \cong F \dotplus J(A)$ and $A_{0} \not\subset Z(A)$. Then $Z(A) \cong F \dot+ J'$, where $J' = J(A) \cap Z(A)$. Moreover, if $J'$ is a graded ideal of $A$, then $C^{gr}(A)$ is a $T_{2}$-ideal of $F \langle Y,Z \rangle$.
\end{pr}

\begin{proof} Let $a \in Z(A)$. Since $A \cong F \dot+ J(A)$, we can write $a = \alpha + j$, with $\alpha \in F, \, j \in J(A)$. As $F \subset Z(A)$, it follows that, $j = a - \alpha \in Z(A)$ and hence $j \in J(A) \cap Z(A) = J'$. This establishes the decomposition $Z(A) \cong F \dot+ J'$.

 Now, suppose \( J' \) is a graded ideal of \( A \) and let \( f \in C^{gr}(A) \). To prove \( C^{gr}(A) \) is a \( T_2 \)-ideal, it suffices to show that for any \( x', x'' \in Y \cup Z \), the polynomial \( g = x' f x'' \) belongs to \( C^{gr}(A) \).

Write \( f \) in the form $f = \alpha y_{1} \cdots y_{n} + \text{cl}(w_{1},\ldots,w_{n})$ and take $a_{0} \in A_{0} \setminus Z(A)$ where $\text{cl}(w_{1},\ldots,w_{n})$, $1 \leq i \leq n$,  is a linear combination of elements of the form
$w_{i_{1}}\cdots w_{i_{t}}c_{j_{1}}\cdots c_{j_{s}},$ with
$w_{i} \in \{y_{1},z_{i}\}$ and $c_{j}$ a Lie commutator in the variables of $Y \cup Z$. Consider the evaluation $\phi$ defined by $y_{1} = a_{0}$, $y_{j} = 1$, for all $j \neq 1$, $z_{k} = 0$, for all $1 \leq k \leq n$. Under this evaluation, $\phi(f) = \alpha a_{0} \in Z(A)$ and since $a_{0} \notin Z(A)$, we have $\alpha = 0$. 

Thus, \( f \) is a linear combination of elements of the form \( w_{i_{1}} \cdots w_{i_{t}} c_{j_{1}} \cdots c_{j_{s}} \), where \( w_i \in \{y_i, z_i\} \) and each \( c_j \) is a Lie commutator.
This implies that for any evaluation $\psi$, we have
 $\psi(w_{i_{1}}w_{i_{2}} \cdots w_{i_{t}}c_{j_{1}}c_{j_{2}} \cdots c_{j_{s}}) \in J(A)$  and so, $\psi(f) \in J(A) \cap Z(A) = J'$. By hypothesis, $J'$ is a graded ideal of $A$, so $a  \psi(f)  b \in J'\subset Z(A)$, for any $a,b \in A_{0} \cup A_{1}$ and therefore we conclude that $g \in C^{gr}(A)$.

\end{proof}

Due to \cite{DrenGi}, it is well known that, if $1 \in A$, $\Idg(A)$ is a $T_{2}$-ideal completely determined by its multilinear proper $\Z_{2}$-polynomials. Recall that $f(y_{1},\ldots,y_{n}, z_{1}, \ldots, z_{n}) \in F\langle Y,Z \rangle$ is a proper $\Z_{2}$-polynomial if it is a linear combination of polynomials of the form $z_{1}\cdots z_{k}w_{1}\cdots w_{s}$, where $w_{i}$ are commutators of arbitrary weight in the variables $y_{1}, \ldots , y_{n},z_{1}, \ldots z_{n}$, for all $1 \leq i \leq s$. Therefore, we have the following result.

\begin{teo}\label{proprer} Let $A$ be a unitary superalgebra such that $A_{0} \not\subset Z(A)$. Consider $Z(A) \cong F \dot+ J'$, $J' = J(A) \cap Z(A)$. If $J'$ is a graded ideal of $A$, so $C^{gr}(A)$ is generated by proper $\Z_{2}$-polynomials.     
\end{teo}

 It is well known that in characteristic zero, every graded identity is equivalent to a system of multilinear graded identities. The same holds for central graded polynomials. For $n\ge 1,$ we denote by

$$P_{n}^{gr} = \spn \{w_{\sigma(1)} \cdots w_{\sigma(n)}) \ | \ \sigma \in S_{n} , w_{i} = y_{i} \ \mbox{or} \ w_{i} = z_{i}, i = 1, \ldots ,n\}$$
the space of multilinear polynomials of degree $n$ in the variables $y_{1}, z_{1}, \ldots , y_{n}, z_{n}$, and we define
$$
 P_{n}^{gr}(A)=\displaystyle\frac{P_{n}^{gr}}{P_{n}^{gr} \cap \Idg(A)}, \ \ 
P_{n}^{gr,z}(A)= \displaystyle\frac{P_{n}^{gr}}{P_{n}^{gr} \cap C^{gr}(A)} \ \ \mbox{and} \ 
\Delta_n^{gr}(A)= \displaystyle\frac{P_n^{gr}\cap C^{gr}(A)}{P_{n}^{gr} \cap \Idg(A)}.
$$

The non-negative integers 
$$c_{n}^{gr}(A) = \dim P_{n}^{gr}(A), \ c_{n}^{gr,z}(A) = \dim P_{n}^{gr,z}(A) \ \mbox{and} \ \delta_{n}^{gr}(A) = \dim \Delta_{n}^{gr}(A)$$ 
are called the $n$th graded codimension,  the $n$th central graded codimension and the $n$th proper central graded codimension of $A$, respectively. Similarly to the ordinary case, we also have
\begin{equation}\label{eqcodim}
c_{n}^{gr}(A) = c_{n}^{gr,z}(A) + \delta_{n}^{gr}(A).
\end{equation}

Other important numerical invariants in PI-theory are the so-called PI-exponent, central exponent  and proper central exponent of $A$, defined respectively as follows:
$$\exp(A):=\lim_{n\to \infty}\sqrt[n]{c_n(A)}, \ \displaystyle \exp^z(A):=\lim_{n\to \infty}\sqrt[n]{c^z_n(A)}
\;\;\mbox{ and \;}\exp^\delta(A):=\lim_{n\to \infty}\sqrt[n]{\delta_n(A)}.$$

The existence and integrality of these invariants were proven by Giambruno and Zaicev \cite{GZ98,GZ99,GZcentasso2,GZcentasso}. In \cite{GiLaMi}, Giambruno, La Mattina and Polcino Milies classified the 
 varieties of algebras having almost polynomial $\delta$-growth. We recall that such a variety $\mathcal{V}$  satisfies $\exp^{\delta}(\mathcal{V})\geq 2$ and for any proper subvariety $\mathcal{W}$ of  we have that
 $\exp^{\delta}(\mathcal{W})\leq 1$. Considering the subalgebra $D=F(e_{11}+e_{33})\oplus Fe_{12}\oplus Fe_{13} \oplus Fe_{22}\oplus Fe_{23}$ of the algebra of $UT_3$ and the subalgebra $D_0= Fe_{12}\oplus Fe_{13} \oplus Fe_{14} \oplus Fe_{22}\oplus Fe_{23} \oplus  Fe_{24} \oplus Fe_{33}\oplus Fe_{34}$ of $UT_4$, the authors proved that the only varieties of almost polynomial $\delta$-growth are $\mathcal{G}$, $D$ and $D_0$. They also presented a similar result in the context of algebras graded by a finite abelian group $G$. In particular, for $G=\Z_2$, the results presented here contribute information about the subvarieties of varieties of superalgebras with almost polynomial $\delta$-growth.

  We observe that the hyperoctahedral group $\mathbb{Z}_2 \wr S_n$ defines a natural action on $P_n^{gr}$ as follows: for $k = (a_1, \ldots, a_n; \sigma) \in \mathbb{Z}_2 \wr S_n$, let
\[
k y_i = y_{\sigma(i)} \quad \text{and} \quad k z_i = 
\begin{cases}
z_{\sigma(i)} & \text{if } a_{\sigma(i)} = 1, \\
-z_{\sigma(i)} & \text{if } a_{\sigma(i)} = -1.
\end{cases}
\]
 Then $P_n^{gr}$ becomes a $\mathbb{Z}_2 \wr S_n$-module. Since $P_n^{gr} \cap \operatorname{Id}^{gr}(A)$ is invariant under this action, the quotient space $P_n^{gr}(A)$ inherits the structure of a $\mathbb{Z}_2 \wr S_n$-module, and we may consider its $\mathbb{Z}_2 \wr S_n$-character $\chi_n^{gr}(A)$, called the $n$th graded cocharacter of $A$. The spaces $P_n^{gr,z}(A)$ and $\Delta_n^{gr}(A)$ also carry a $\mathbb{Z}_2 \wr S_n$-module structure, and we denote their characters by $\chi_n^{gr,z}(A)$ and $\chi_{n}(\Delta^{gr}(A))$, called the $n$th central graded cocharacter and the $n$th graded proper central cocharacter of $A$, respectively.

Recall that the irreducible $\mathbb{Z}_2 \wr S_n$-characters are in one-to-one correspondence with pairs of partitions $(\lambda, \mu)$, where $\lambda \vdash n - r$ and $\mu \vdash r$ for $r = 0, 1, \ldots, n$. By complete reducibility, we have the decompositions:

\[
\chi_n^{gr}(A) = \sum_{|\lambda| + |\mu| = n} m_{\lambda,\mu} \chi_{\lambda,\mu}, \quad
\chi_n^{gr,z}(A) = \sum_{|\lambda| + |\mu| = n} m'_{\lambda,\mu} \chi_{\lambda,\mu}, \quad
  \chi_{n}(\Delta^{gr}(A)) = \sum_{|\lambda| + |\mu| = n} m''_{\lambda,\mu} \chi_{\lambda,\mu},
\]
where $\chi_{\lambda,\mu}$ is the irreducible $\mathbb{Z}_2 \wr S_n$-character associated with the pair $(\lambda, \mu) \vdash n$. The multiplicities satisfy $m_{\lambda,\mu} = m'_{\lambda,\mu} + m''_{\lambda,\mu}$, and we write
\begin{equation}\label{eqcocar}
\chi_n^{gr}(A) = \chi_n^{gr,z}(A) + \chi_{n}(\Delta^{gr}(A)).
\end{equation}

  The degree of the irreducible character $\chi_{\lambda,\mu}$ is denoted by $d_{\lambda,\mu}$ and is given by
\[
d_{\lambda,\mu} = \binom{n}{r} d_\lambda d_\mu,
\]
where $d_\lambda$ and $d_\mu$ are the degrees of the irreducible $S_{n-r}$- and $S_r$-characters $\chi_\lambda$ and $\chi_\mu$, respectively, given by the hook formula (see \cite{hook}). Since $c_n^{gr}(A) = \chi_n^{gr}(A)(1)$ and $c_n^{gr,z}(A) = \chi_n^{gr,z}(A)(1)$ for all $n \geq 1$, it follows from the above decompositions that
\[
c_n^{gr}(A) = \sum_{|\lambda| + |\mu| = n} m_{\lambda,\mu} d_{\lambda,\mu}, \quad \text{and} \quad
c_n^{gr,z}(A) = \sum_{|\lambda| + |\mu| = n} m'_{\lambda,\mu} d_{\lambda,\mu}.
\]

 Another approach to determining the cocharacter decomposition is to consider the action of $S_{n-r} \times S_r$. For $0 \leq r \leq n$, let $P_{n-r,r}^{gr}$ denote the space of multilinear polynomials in the variables $y_1, \ldots, y_{n-r}, z_{n-r+1}, \ldots, z_n$. To study $P_n^{gr} \cap \operatorname{Id}^{gr}(A)$, it suffices to study $P_{n-r,r}^{gr} \cap \operatorname{Id}^{gr}(A)$ for all $r \geq 0$, which can be done via the representation theory of $S_{n-r} \times S_r$. Define
\[
P_{n-r,r}^{gr}(A) = \frac{P_{n-r,r}^{gr}}{P_{n-r,r}^{gr} \cap \operatorname{Id}^{gr}(A)}, \quad
P_{n-r,r}^{gr,z}(A) = \frac{P_{n-r,r}^{gr}}{P_{n-r,r}^{gr} \cap C^{gr}(A)},
\]
and let
\[
c_{n-r,r}^{gr}(A) = \dim P_{n-r,r}^{gr}(A), \quad
c_{n-r,r}^{gr,z}(A) = \dim P_{n-r,r}^{gr,z}(A).
\]
Then we have
\begin{equation}\label{codimenr}
c_n^{gr}(A) = \sum_{r=0}^n \binom{n}{r} c_{n-r,r}^{gr}(A), \quad \text{and} \quad
c_n^{gr,z}(A) = \sum_{r=0}^n \binom{n}{r} c_{n-r,r}^{gr,z}(A).
\end{equation}

Moreover, we may consider the permutation action of $S_{n-r}$ on the variables $y_1, \ldots, y_{n-r}$ and of $S_r$ on the variables $z_{n-r+1}, \ldots, z_n$. These induce a (left) action of $S_{n-r} \times S_r$ on $P_{n-r,r}^{gr}$. Since $T_2$-ideals (and $T_2$-spaces) are invariant under permutations of symmetric (respectively skew) variables, it follows that $P_{n-r,r}^{gr}(A)$ and $P_{n-r,r}^{gr,z}(A)$ inherit the structure of left $S_{n-r} \times S_r$-modules. We denote their characters by $\chi_{n-r,r}^{gr}(A)$ and $\chi_{n-r,r}^{gr,z}(A)$, respectively.

\section{The varieties $\textnormal{var}^{gr}(D^{gr}), \ \textnormal{var}^{gr}(UT_{2})\ \ \mbox{and} \ \ \textnormal{var}^{gr}(UT_{2}^{gr})$}

We are interested in studying the behavior of the central polynomials of superalgebras that generate minimal subvarieties of the almost polynomial growth supervarieties. We investigate the central codimensions and the central cocharacter of these superalgebras, comparing the  obtained results. This section presents the main results about the %
varieties $ \text{var}^{gr}(D^{gr}), \text{var}^{gr}(UT_2) \ \mbox{and} \  \text{var}^{gr}(UT_2^{gr})$ as well as the minimal subvarieties contained therein.

Firstly, we consider the superalgebra $D^{gr}$. We already know that its $T_2$-ideal is $\langle [y_{1},y_{2}],[z_{1},z_{2}],[y_1,z_2]\rangle_{T_{2}}$ (see \cite{GiMi}). Since $D^{gr}$ is a commutative superalgebra, we conclude that its $T_2$-space is $F\langle Y,Z \rangle$, except for non-zero constant polynomials. As a consequence, we have $c_{n}^{gr,z}(D^{gr}) = 0$, $\delta_{n}^{gr,z}(D^{gr}) = c_{n}^{gr}(D^{gr}) = 2^{n}$, $\chi_{n}^{gr,z}(D^{gr}) = 0$ and $\chi_{n}(\Delta^{gr}(D^{gr})) = \chi_{n}^{gr}(D^{gr}) = \ds\sum_{i=0}^{n} \chi_{(n-i),(i)}$ (see \cite{GiMi}, \cite{GiaMiZai}).  

To describe the minimal varieties in $\text{var}^{gr}(D^{gr})$
classified in \cite{LaMatalmostGgr}, for  $k \geq 2$, let $I_{k}$ denote the $k \times k$ identity matrix and $E = \displaystyle\sum_{i=1}^{k-1} e_{i,i+1} \in M_{k}(F)$. Consider $C_k^{gr}$ to be the subalgebra of $UT_{k}$
$$C_k = \spn\{I_{k},E, \ldots, E^{k-1}\}$$
with elementary grading induced by $\textbf{g}=(0,1,0,1,...) \in \mathbb{Z}_2^{k}.$

By \cite[Corollary 8.2]{LaMatalmostGgr},  a superalgebra $A$ generates a minimal variety in $\text{var}^{gr}(D^{gr})$ if and only if $\Idg(A) = \Idg(C_{k}^{gr})$ for some  $k \geq 2.$ Notice that $C_{k}$ is a commutative subalgebra of $UT_{k}$. Hence, every polynomial in $F \langle Y,Z \rangle$ (except non-zero constant polynomials) is a central polynomial of $C_{k}^{gr}$. Therefore, we immediately conclude that $c_{n}^{gr,z}(C_{k}) = 0$, $\delta_{n}^{gr}(C_{k}) = c_{n}^{gr}(C_{k}^{gr}) = \displaystyle\sum_{i=0}^{k-1} \binom{n}{i}$, $\chi_{n}^{gr,z}(C_{k}^{gr}) = 0$ and $\chi_{n}(\Delta^{gr}(C_{k}^{gr})) = \displaystyle\sum_{i=0}^{k-1} \chi_{(n-i),(i)}$ (see \cite{LaMaMi}, \cite{NaSaVi}).

Now, we shall carry out a detailed study of the behavior of the central polynomials of the minimal subvarieties contained in the variety $\text{var}^{gr}(UT_2)$. 
Note that $UT_{2}$ has no proper central graded polynomials, so we conclude that $C^{gr}(UT_{2}) = \Idg(UT_{2}) = \langle [y_{1},y_{2}][y_{3},y_{4}],z \rangle_{T_{2}}$, $c_{n}^{gr,z}(UT_{2}) = c_{n}^{gr}(UT_{2}) = 2^{n-1}(n-2)+2 $, $\delta_n^{gr}(UT_{2}) = 0$, $\chi_{n}^{gr,z}(UT_{2}) = \chi_{n}^{gr}(UT_{2}) = \chi_{(n), \varnothing} + \displaystyle\sum_{i\ge 1} (n-2i+1) \chi_{(n-i, i),\varnothing} + \displaystyle\sum_{i\ge 1} (n-2i) \chi_{(n-i-1, i, 1),\varnothing}$ and $\chi_{n}(\Delta^{gr}(UT_{2}))  = 0$ (see \cite{BeGiSvi}, \cite{MiReZai}).

In order to present the classification of the minimal subvarieties of $\textnormal{var}^{gr}(UT_{2})$, we denote by $A_{k},B_{k}, N_{k}$ the following subalgebras of $UT_{k}$, $k \geq 2$:

$$A_{k} = \spn\{e_{11}, E, E^{2}, \ldots, E^{k-2},e_{12},e_{13},\ldots,e_{1k}\},$$
$$B_{k} = \spn\{e_{kk},E, E^{2}, \ldots, E^{k-2},e_{1k},e_{2k},\ldots,e_{k-1,k}\}, $$
$$ N_{k}=\spn\{I_k,E, E^{2}, \ldots , E^{k-2},e_{12},e_{13},...,e_{1k}\}$$
where $E = \ds\sum_{i=1}^{k-1} e_{i,i+1} \in UT_{k}$ and $I_k$ denotes the $k\times k$ identity matrix. 

Consider $A_{k}, B_{k}$ and $N_{k}$ with the trivial grading. By \cite[Corollary 5.4]{LaMat},  a superalgebra $A$ generates a minimal variety in $\text{var}^{gr}(UT_{2})$ if and only if $\Idg(A) = \Idg(A_{k})$ or $\Idg(A) = \Idg(B_{k})$ or $\Idg(A) = \Idg(N_{t})$, for some  $k \geq 2, t > 2$. 

For the superalgebras $A_{k}$ and  $B_{k}$, the $T_2$-ideal, the codimensions, and the cocharacter decomposition are already known (see \cite{GutReg}, \cite{LaMat}, \cite{NaSaVi}). Since the superalgebras $A_{k}$ and  $B_{k}$ have trivial center, we can conclude the following.

\begin{teo} For $k \geq 2$, we have

\begin{enumerate}
    \item $C^{gr}(A_k) = \Idg(A_{k}) = \langle [y_{1},y_{2}][y_{3},y_{4}], [y_{1},y_{2}] y_{3} y_{4} \cdots y_{k+1},z\rangle_{T_{2}}$;
    \item $C^{gr}(B_k)=   \Idg(B_{k}) = \langle [y_{1},y_{2}][y_{3},y_{4}], y_{3} y_{4} \cdots y_{k+1}[y_{1},y_{2}] ,z\rangle_{T_{2}}$;
    \item $c_{n}^{gr,z}(A) = c_{n}^{gr}(A)  =1+\displaystyle\sum_{i=0}^{k-2} \binom{n}{i} (n-i-1)$, for $A \in \{A_k, B_k\}$;
    \item $\chi_{n}^{gr,z}(A) = \chi_{n}^{gr}(A) = \chi_{(n), \varnothing} + \ds\sum_{i=1}^{k-1}(k- i) \chi_{(n-i,i),\varnothing} + \ds\sum_{i=1}^{k-2} (k-i-1) \chi_{(n-i-1,i,1),\varnothing}$, for $A \in \{A_k, B_k\}$;
    \item $\delta_n^{gr}(A)= 0$ and $\chi_{n}(\Delta^{gr}(A)) = 0$, for $A \in \{A_k, B_k\}$.
    
\end{enumerate} 
    
\end{teo}

Now, for the superalgebra $N_{k}$, $k \geq 2$. Firstly, notice that $N_{2} = C_{2}$ and so this case was solved. For $k \geq 3$, we have the following.

\begin{teo} (\cite[Theorem 3.4]{GLMPetro} and \cite[Theorem 5.5]{NaSaVi}) Let $k \geq 3$. Then:

\begin{enumerate}
    \item $\Idg(N_{k}) = \langle [y_{1},\ldots,y_{k}],[y_{1},y_{2}][y_{3},y_{4}],z\rangle_{T_{2}}$;
    \item $c_{n}^{gr}(N_{k}) = 1+ \displaystyle\sum_{i=2}^{k-1} \binom{n}{i} (i-1)$;
    \item $\chi_{n}^{gr}(N_{k}) = \chi_{(n),\varnothing} + \ds\sum_{i=1}^{k-2}(k-i-1)(\chi_{(n-i,i),\varnothing} + \chi_{(n-i-1,i,1),\varnothing})$.
\end{enumerate}
    
\end{teo}



It is clear that the center of the algebra $N_{k}$ is given by
$Z(N_k) = \spn\{I_k, e_{1k}\}$.
Consequently, for $k\geq 3$, the algebras $N_k$ satisfy the hypotheses of Theorem \ref{proprer}, then $C^{gr}(N_k)$ is a $T_{2}$-ideal generated by proper $\Z_{2}$-polynomials. Our current goal is to determine the $T_{2}$-space $C^{gr}(N_{k})$.

\begin{teo}\label{t-spaceNk} Let $k \geq 3$. Then

$$C^{gr}(N_{k}) = \Idg (N_{k-1}) = \langle [y_{1},\ldots,y_{k-1}],[y_{1},y_{2}][y_{3},y_{4}],z\rangle_{T_{2}}.$$

\end{teo}

\begin{proof} Consider $I = \left\langle \left[ y_{1} , \ldots ,y_{k-1}  \right],[y_{1},y_{2}][y_{3},y_{4}],z \right\rangle _{T_2}$. We can see that $I \subset C^{gr}(N_{k})$, since $[y_{1},y_{2}][y_{3},y_{4}] \in \Idg (N_k)$ and $[\underbrace{N_k , \ldots ,N_k}_{k-1}] \subset \mbox{span}\{e_{1k}\}$. 

We now prove the reverse inclusion. Let $f \in I$ be a multilinear $\Z_{2}$-polynomial of degree $s$. Since $N_{k}$ is a unitary algebra and $N_k$ satisfies the conditions in Theorem \ref{proprer}, we may assume $f$ is a proper $\Z_{2}$-polynomial. After reducing $f$ modulo $I$,  we obtain the following:

(i) If $s \geq k-1$, we have $f \equiv 0$ in $N_k$,

(ii) If $s \leq k-2$, so $f$ is a linear combination of polynomials of the type
$$[y_{i}, y_{i_1}, \ldots , y_{i_{s-1}}],
$$
where $i = 2, \ldots , s$, $i_1< \ldots < i_{s-1}$, that is, 

$$f = \sum\limits_{i = 2}^s {\alpha _i [y_{i} ,y_{i_1} , \ldots ,y_{i_{s - 1}} ]}, \ \alpha_{i} \in F, \ i \in \{2,\ldots,s\}.$$

Suppose that there exists $i$ such that $\alpha _i \neq 0$. By making the evaluation  $\phi$ given by   $y_{i} = e_{12}$ and $y_{j} = E,$ for all $j \neq i$, we get $\phi(f)=\alpha_i e_{1, s+1}$. This is a contradiction, since $e_{1, s+1} \notin Z(N_k)$ for $1\le s \le k-2$. Then, $\alpha _i = 0,$ for all $2 \le i \le s$. Hence $C^{gr}(N_{k})=\Idg (N_{k-1})$.




\end{proof}

As a consequence of the previous theorem  and of equations (\ref{eqcodim}) and (\ref{eqcocar}) we have the next corollary.

\begin{cor} For $k \geq 3$, we have:

\begin{enumerate}
 \item $c_{n}^{gr,z}(N_{k}) = c_{n}^{gr}(N_{k-1}) = 1+ \displaystyle\sum_{i=2}^{k-2} \binom{n}{i} (i-1)$;

 \item $\delta_{n}^{gr}(N_{k}) = \ds\binom{n}{k-1} (k-2)$;
    \item $\chi_{n}^{gr,z}(N_{k}) = \chi_{n}^{gr}(N_{k-1}) = \chi_{(n),\varnothing} + \ds\sum_{i=1}^{k-3}(k-i-2)(\chi_{(n-i,i),\varnothing} + \chi_{(n-i-1,i,1),\varnothing})$
    \item $\chi_{n}(\Delta^{gr}(N_{k})) = \ds\sum_{i=1}^{k-2}(\chi_{(n-i,i),\varnothing} + \chi_{(n-i-1,i,1),\varnothing}).$
\end{enumerate}

\end{cor}

The last variety considered in this section is the variety $\text{var}^{gr}(UT_2^{gr})$.  Since $UT_{2}^{gr}$ has no proper central polynomials, from the known results, we conclude that $C^{gr}(UT_{2}^{gr}) = \Idg(UT_{2}^{gr}) = \langle z_1z_2 \rangle_{T_{2}}$, $c_{n}^{gr,z}(UT_{2}^{gr}) = c_{n}^{gr}(UT_{2}^{gr}) = 1+n2^{n-1} $, $\delta_n^{gr}(UT_{2}^{gr}) = 0$, $\chi_{n}^{gr,z}(UT_{2}^{gr}) = \chi_{n}^{gr}(UT_{2}^{gr}) = \chi_{(n), \varnothing}  + \displaystyle\sum_{i\ge 1} (n-2i) \chi_{(n-i-1, i), (1)}$ and $\chi_{n}(\Delta^{gr}(UT_{2}^{gr}))  = 0$ (see \cite{Va}).

Now, we are interested in examining the minimal subvarieties in $\textnormal{var}^{gr}(UT_{2}^{gr})$. Denote by $A_{k}^{gr}$, $B_{k}^{gr}$ and $N_{k}^{gr}$ the algebras $A_{k}$, $B_{k}$ and  $N_{k}$ with elementary grading induced by $(0,1,\ldots,1)$, respectively. By \cite[Corollary 6.1]{LaMatalmostGgr} we have that $A$ is a superalgebra generating a minimal variety in $\text{var}^{gr}(UT_{2}^{gr})$ if and only if $\Idg(A) = \Idg(A_{k}^{gr})$ or $\Idg(A) = \Idg(B_{k}^{gr})$ or $\Idg(A) = \Idg(N_{k}^{gr})$, for some  $k \geq 2$. 

In the case of the superalgebras  $A_{k}^{gr}$ and $B_{k}^{gr}$, the $T_{2}$-ideal, the sequence of the graded codimension, and the decomposition of the graded cocharacter have already been determined (see \cite{LaMatalmostGgr}, \cite{NaSaVi}). Once again we use the fact that these algebras have trivial center, and we conclude the following.

\begin{teo} For $k \geq 2$
\begin{enumerate}
    \item $C^{gr}(A_k^{gr})=\Idg(A_{k}^{gr}) = \langle [y_{1},y_{2}],z_{1}y_{2}\cdots y_{k},z_{1}z_{2} \rangle_{T_{2}}$;

    \item $C^{gr}(B_k^{gr})=\Idg(B_{k}^{gr}) = \langle [y_{1},y_{2}],y_{2}\cdots y_{k}z_{1},z_{1}z_{2} \rangle_{T_{2}}$;

    \item $c_{n}^{gr,z}(A) = c_{n}^{gr}(A) = 1 + \ds\sum_{i=0}^{k-2} \binom{n}{i} (n-i)$, for $A\in \{A_k^{gr}, B_k^{gr}\}$;

    \item $\chi_{n}^{gr, z}(A) =\chi_{n}^{gr}(A) = \chi_{(n),\varnothing} + \ds\sum_{i=0}^{k-2}(k-i-1) \chi_{(n-i-1,i),(1)}$, for $A\in \{A_k^{gr}, B_k^{gr}\}$;
    \item $\delta_n^{gr}(A)= 0$ and $\chi_{n}(\Delta^{gr}(A)) = 0$, for $A\in \{A_k^{gr}, B_k^{gr}\}$ .

\end{enumerate}

\end{teo}

Finally, for the superalgebra $N_{k}^{gr}$, we have the following result about the sequence of graded codimensions and graded cocharacters.

\begin{teo}(\cite[Theorem 4.1]{LaMatalmostGgr} and \cite[Theorem 6.4]{NaSaVi}) Let $k \geq 2$. Then:

\begin{enumerate}
    \item $\Idg(N_{k}^{gr}) = \langle [y_{1},y_{2}], [z, y_{1},\ldots,y_{k-1}],z_{1}z_{2} \rangle_{T_{2}}$;
    \item $c_{n}^{gr}(N_{k}^{gr}) = 1+ \displaystyle\sum_{i=1}^{k-1} \binom{n}{i} i$;
    \item $\chi_{n}^{gr}(N_{k}^{gr}) = \chi_{(n),\varnothing} + \ds\sum_{i=0}^{k-2}(k-i-1) \chi_{(n-i-1,i),(1)}$.
\end{enumerate}

\end{teo}

Our current goal is to determine the $T_{2}$-space $C^{gr}(N_{k}^{gr})$.

\begin{teo}\label{t-spaceNkgr} Let $k \geq 3$. Then

$$C^{gr}(N_{k}^{gr}) = \Idg (N_{k-1}^{gr}) =  \langle [y_{1},y_{2}], [z, y_{1},\ldots,y_{k-2}],z_{1}z_{2} \rangle_{T_{2}}.$$

\end{teo}
\begin{proof}  Let $I = \langle [y_{1},y_{2}], [z, y_{1},\ldots,y_{k-2}],z_{1}z_{2} \rangle _{T_2}$. We notice that $z_{1}z_{2} \in \Idg (N_k^{gr})$ and $[N_k^{(1)} ,\underbrace{N_k^{(0)} , \ldots ,N_k^{(0)}}_{k-2}] \subset \mbox{span}\{e_{1k}\}$. Since $Z(N_k)$ is a graded ideal of $N_k$, we have $I \subset C^{gr}(N_{k}^{gr})$.

Next, we shall prove the opposite inclusion. Let $f \in I$ be a multilinear polynomial of degree $s$. Again, since $N_{k}$ is a unitary algebra and satisfies the conditions in  Theorem \ref{proprer}, we can assume that $f$ is a proper $\Z_2$-polynomial. After reducing $f$ modulo $I$,  we obtain the following:

(i) If $s \geq k-1$, we have $f \equiv 0$.

(ii) If $s \le k-2$, then $f$ is a linear combination of polynomials
$$[z_{i}, y_{i_1}, \ldots , y_{i_{s-1}}],
$$
where $i_1< \ldots < i_{s-1}$ and $1\le i\le n$.

Therefore, modulo $I$, we can assume that, for $1\le s \le k-2$,
$$
f = \sum\limits_{i = 1}^s {\alpha _i [z_{i} ,y_{i_1} , \ldots ,y_{i_{s - 1}} ]}.
$$

Suppose that there exists $i$ such that $\alpha _i \neq 0$. By making the evaluation  $\phi$ given by $z_{i} = e_{12}, z_j=0$ for $j\neq i$, and $y_{i_m} = E,$ for all $m = i_1, \ldots , i_{s-1}$, we get  $\phi(f)=\alpha_i e_{1 \ s+1}$. This is a contradiction, since $e_{1 \ s+1} \notin Z(N_k)$ for $1\le s \le k-2$. Then, $\alpha _i = 0,$ for all $1 \le i \le s$. Therefore $C^{gr}(N_{k}^{gr})=\Idg (N_{k-1}^{gr})$.

\end{proof}

As a consequence of Theorem \ref{t-spaceNkgr} and of equations (\ref{eqcodim}) and (\ref{eqcocar}), we obtain the following.

\begin{cor} For $k \geq 3$, we have:

\begin{enumerate}
 \item $c_{n}^{gr,z}(N_{k}^{gr}) = c_{n}^{gr}(N_{k-1}^{gr}) = 1 + \displaystyle\sum_{i=1}^{k-2} \binom{n}{i} i$;

 \item $\delta_{n}^{gr}(N_{k}^{gr}) = \ds\binom{n}{k-1} (k-1)$;

\item $\chi_{n}^{gr,z}(N_{k}^{gr}) = \chi_{n}^{gr}(N_{k-1}^{gr}) = \chi_{(n),\varnothing} + \ds\sum_{i=0}^{k-3}(k-i-2) \chi_{(n-i-1,i),(1)}$

\item $\chi_{n}(\Delta^{gr}(N_{k}^{gr})) = \ds\sum_{i=0}^{k-2} \chi_{(n-i-1,i),(1)}$.
\end{enumerate}

\end{cor}


\section{The variety $\textnormal{var}^{gr}(\mathcal{G})$}

In this section, we are interested in examining the minimal subvarieties in $\textnormal{var}^{gr}(\mathcal{G})$. We recall that for the superalgebra $\mathcal{G}$, we have $\Idg(\mathcal{G})=\langle [y_1,y_2,y_3],z\rangle_{T_2}$, $C^{gr}(\mathcal{G})=\langle [y_1,y_2], y_0[y_1,y_2,y_3],z\rangle^{T_2}$, $c_n^{gr}(\mathcal{G})=2^{n-1}, c_n^{gr,z}(\mathcal{G})=2^{n-2}$, $\delta_n^{gr}(\mathcal{G})=2^{n-2}$, $\chi_n^{gr}(\mathcal{G})=\ds\sum_{\substack{ 1\leq i < n}} \chi_{(n-i,1^{i}),\varnothing}$, $\chi_n^{gr,z}(\mathcal{G})=\ds\sum_{\substack{ 1\leq i \leq n \\ i \text{ even} }} \chi_{(n-i,1^{i}), \varnothing}$ and finally $\chi_{n}(\Delta^{gr}(\mathcal{G}))=\ds\sum_{\substack{ 1\leq i \leq n \\ i \text{ odd} }} \chi_{(n-i,1^{i}), \varnothing}$ (see  \cite{BrKoKrSi}, \cite{OlReg}, \cite{reggrow} ).

Now we consider the finite-dimensional Grassmann algebra: 
$$\mathcal{G}_{t} = \langle 1, e_{1}, \ldots, e_{t} | e_{i}e_{j} = - e_{j}e_{i} \rangle $$
with the trivial grading. We recall that $\Idg(\mathcal{G}_{2k}) = \Idg(\mathcal{G}_{2k+1})$ and by \cite[Corollary 5.3]{LaMat} we have that a superalgebra $A$ generates a minimal variety in $\textnormal{var}^{gr}(\mathcal{G})$ if and only if $\Idg(A) = \Idg(\mathcal{G}_{2k})$ for some  $k \geq 1$. 

For such superalgebras, we have the following.

\begin{teo}(\cite[Theorem 3.5]{GLMPetro} and \cite[Theorem 4.3]{NaSaVi})\label{codg2k}
Let $k\in \N$. 
\begin{enumerate}
\item $\Id^{gr}(\mathcal{G}_{2k})= \langle [y_1,y_2,y_3],
[y_1,y_2][y_3,y_4]\dots[y_{2k+1},y_{2k+2}],z\rangle_{T_{2}}$,   
\item
$ c^{gr}_n(\mathcal{G}_{2k}) = \ds\sum_{i=0}^{k} \binom{n}{2i} $,

\item $\chi_{n}^{gr}(\mathcal{G}_{2k}) = \ds\sum_{i=0}^{2k} \chi_{(n-i,1^{i}),\varnothing}$. 
\end{enumerate}
\end{teo}

Since $\mathcal{G}_{2k}$ and $\mathcal{G}_{2k+1}$
satisfy the same graded identities, it follows from Remark \ref{rie} that they also share the same central graded polynomials. Concerning the study of central graded polynomials for the superalgebras $\mathcal{G}_{2k}$, the corresponding $T_2$-space has already been determined. 
\begin{teo} (\cite[Proposition 6]{KoKraSil})\label{G2k} For $k\geq 1,$
$$C^{gr}(\mathcal{G}_{2k})=\langle [y_{1},y_{2}], y_0[y_1,y_2,y_3],y_0[y_1,y_2][y_3,y_4]\dots[y_{2k-1},y_{2k}],z\rangle^{T_{2}}.$$
\end{teo}

 We now present the results for $c_{n}^{gr,z}(\mathcal{G}_{2k})$ and $\chi_{n}^{gr,z}(\mathcal{G}_{2k})$. To achieve our goal, we need the following technical remark.

{\begin{obs}{\label{comutadores}} \hfill

\begin{enumerate}
    \item $y_{2}\cdots y_{n-1}[y_1,y_n] \equiv -\ds y_1\sum_{i=2}^{n-1} y_2\cdots y_{i-1}y_{i+1}\cdots y_{n-1}[y_i,y_n]$ mod $\langle [y_1,y_2], y_0[y_1,y_2,y_3]\rangle^{T_{2}}$.
    \item Let $\sigma \in S_{2p}$, $p \in \N$. Then
$$[y_{\sigma(1)},y_{\sigma(2)}]\cdots [y_{\sigma(p-1)},y_{\sigma(2p)}] \equiv \text{sgn}(\sigma)[y_{1},y_{2}]\cdots[y_{{2p-1}},y_{{2p}}]
\mod \langle [y_1,y_2,y_3]\rangle_{T_{2}}.$$
\end{enumerate}

\end{obs}}

We are now in a position to prove the following result:

\begin{teo} For $k \in \mathbb{N}$, we have

\begin{enumerate}
    \item $\ds c^{gr,z}_n(\mathcal{G}_{2k})= \ds\sum_{i=0}^{k-1}
\binom{n-1}{2i}$;
 \item $\ds \delta^{gr}_n(\mathcal{G}_{2k})= \ds \binom{n}{2k} + \sum_{i=0}^{k-2}\binom{n-1}{2i+1}$;

   \item $\chi_n^{gr,z} (\mathcal{G}_{2k})= \ds\sum_{i=0}^{k-1} \chi_{(n-2i, 1^{2i}),\varnothing}$;
   \item $\chi_{n}(\Delta^{gr} (\mathcal{G}_{2k}))=\chi_{(n-2k,1^{2k}),\varnothing}+ \ds\sum_{i=0}^{k-1} \chi_{(n-2i-1, 1^{2i+1}),\varnothing}$.
\end{enumerate}

\end{teo}

\begin{proof}

\

First, let $n\geq 2k-1$ and for each $r$ such that $0\le r\le k-1$, consider all possible subsets $J_r=\{ j_1,j_2,\dots,j_{2r}\} \subseteq\{2,3,\dots,n\}$ with  $1<j_1<j_2< \dots < j_{2r}$. Note that for a fixed $r$, there are $n_r:=\ds\binom{n-1}{2r}$ such sets $J_r$. For a fixed subset $J_r$, let $\{ i_1, \dots,i_{n-2r}\}=\{1,2,\dots,n\} \setminus J_r$ with $1=i_1<\dots<i_{n-2r}$ and define the following multilinear polynomial of degree $n$
$$v_{J_r}:=y_{1}y_{i_2}\cdots y_{i_{n-2r}}[y_{j_1},y_{j_2}][y_{j_3},y_{j_4}]\dots[y_{j_{2r-1}},y_{j_{2r}}].$$

 Now let $\tilde{J}_r=\{v_{J_r^1}, \ldots, v_{J_r^{n_r}}\}$ be the set of all possible polynomials $v_{J_r}$ constructed as above for each choice of the subsets $J_r$ and let $\mathcal{B}= \bigcup \tilde{J}_r$ where $r$ runs through values from $0$ to $k-1$. Our goal is to show that $\mathcal{B}$ is a basis of $P_{n}^{gr,z}(\mathcal{G}_{2k})$.

Let $f\in P_{n}^{gr}$. Observing that $C^{gr}(\mathcal{G}_{2k})=\langle [y_{1},y_{2}], y_0[y_1,y_2,y_3],y_0[y_1,y_2][y_3,y_4]\dots[y_{2k-1},y_{2k}],z\rangle^{T_{2}}$, by Theorem \ref{G2k},  we may express $f$ as a linear combination of polynomials of the form 
$$y_{i_1}y_{i_2}\cdots y_{i_{n-2s}}[y_{j_1},y_{j_2}][y_{j_3},y_{j_4}]\dots[y_{j_{2s-1}},y_{j_{2s}}]~ \mbox{mod}~ C^{gr}(\mathcal{G}_{2k})$$ 
where  $i_1<\dots<i_{n-2s}$ e $s\in \{ 0,1,\dots , k-1\} $. By Remark \ref{comutadores} we may assume $j_1<j_2<\dots <j_{2s-1}<j_{2s}$ and conclude $y_{i_1}y_{i_2}\cdots y_{i_{n-2s}}[y_{j_1},y_{j_2}][y_{j_3},y_{j_4}]\dots[y_{j_{2s-1}},y_{j_{2s}}]$ is a linear combination of polynomials in $\mathcal{B}$ modulo $C^{gr}(\mathcal{G}_{2k})$.

 Now, consider 
\begin{equation}\label{eq1}
c_{0}y_{1} \cdots y_{n} + \ds\sum_{i=1}^{n_{1}} c_{1}^{i}v_{J_1^i} + \sum_{i=1}^{n_{2}} c_{2}^{i}v_{J_2^i} + \cdots + \sum_{i=1}^{n_{k-1}} c_{k-1}^{i}v_{J_{k-1}^i} \equiv 0  ~\mbox{mod}~C^{gr}(\mathcal{G}_{2k})
\end{equation}
where $c_0, c_t^i\in F$ for $1\leq t\leq k-1$.

Consider the evaluation $\varphi$ defined by $\varphi(y_{1}) = e_{1}$ and $\varphi(y_{j}) = 1$, for all $j \neq 1$. Applying $\varphi$ to (\ref{eq1}) yields 
$$\varphi \left(c_{0}y_{1} \cdots y_{n} + \ds\sum_{i=1}^{n_{1}} c_{1}^{i}v_{J_1^i} + \sum_{i=1}^{n_{2}} c_{2}^{i}v_{J_2^i} + \cdots + \sum_{i=1}^{n_{k-1}} c_{k-1}^{i}v_{J_{k-1}^i}\right) = c_{0}e_{1} \in Z(\mathcal{G}_{2k})$$ 
which implies $c_{0} = 0$.

Now, suppose that for some $t$, there exists a coefficient $c_{t}^{i_{0}}\neq 0$ where $J_t^{i_{0}}$ is chosen so that $|J_t^{i_{0}}|$ is as small as possible. Then (\ref{eq1}) can be written as
\begin{equation}\label{eq2}
 \ds\sum_{i=1}^{n_{t}} c_{t}^{i}v_{J_t^i} + \cdots + \sum_{i=1}^{n_{k-1}} c_{k-1}^{i}v_{J_{k-1}^i} \equiv 0  ~\mbox{mod}~C^{gr}(\mathcal{G}_{2k})
\end{equation}
where $c_s^i\in F,$ for $t\leq s\leq k-1$.
Consider the evaluation 
\[
\phi:\left\{   
\begin{array}{ll}
y_1=e_1 & \\
y_j=e_j & \text{ if } j\in J_t^{i_{0}}, \\
y_j=1 &  \text{ if } j\in \{2,3,\dots,n\} \setminus J_t^{i_{0}}.
\end{array}
\right. 
\]

Evaluating $\phi$ in (\ref{eq2}), we obtain

$$\phi \left(\ds\sum_{i=1}^{n_{t}} c_{t}^{i}v_{J_t^i} + \cdots + \sum_{i=1}^{n_{k-1}} c_{k-1}^{i}v_{J_{k-1}^i} \right) = c_{t}^{i_{0}} \alpha e_1  \ds\prod_{i\in J_t^{i_{0}}} e_{i} \in Z(\mathcal{G}_{k})$$
where $\alpha = 2^{|J_t^{i_{0}}|/2}$.

Since $e_1\ds \prod_{i\in J_t^{i_{0}}} e_{i}$ is an odd element and $c_{t}^{i_{0}} \alpha e_1\ds \prod_{i\in J_t^{i_{0}}} e_{i} \neq 0$, it follows that $c_{t}^{i_{0}}=0$, which is a contradiction. So we conclude that all coefficients $c_{i}^{t}$ must be zero. 

Hence, $\mathcal{B}$ is a basis for $P_n^{gr,z}(\mathcal{G}_{2k})$. By counting its elements and using Theorem \ref{codg2k} and equation (\ref{eqcodim}), we conclude
\begin{equation}\label{codimcentralgk}\ds c^{gr,z}_n(\mathcal{G}_{2k})= \ds\sum_{i=0}^{k-1}
\binom{n-1}{2i} \  \  \mbox{and} \  \ \ds \delta^{gr}_n(\mathcal{G}_{2k})= \ds \binom{n}{2k} + \sum_{i=0}^{k-2}\binom{n-1}{2i+1}.
\end{equation}

We now present the decomposition of the $n$th central graded and the $n$th graded proper central cocharacter. It is known that $\chi_{n}^{gr}(\mathcal{G}_{2k}) = \ds\sum_{i=0}^{2k} \chi_{(n-i,1^{i}),\varnothing}$ is the decomposition of the $n$th graded cocharacter of $\mathcal{G}_{2k}$. Then by equation (\ref{eqcocar}), we have $\chi_{n}^{gr,z}(\mathcal{G}_{2k}) = \ds\sum_{i=0}^{2k} \alpha_i\chi_{(n-i,1^{i}),\varnothing},$ where  $\alpha_0, \alpha_1, \dots, \alpha_{2k}\in \{0,1\}$. Then, since $c_n^{gr,z} (\mathcal{G}_{2k})= \chi_n^{gr,z} (\mathcal{G}_{2k})(1)$, we have  
$$
c_n^{gr,z} (\mathcal{G}_{2k})= \ds\sum_{i=0}^{2k} \alpha_id_{(n-i,1^{i}),\varnothing} = \ds\sum_{i=0}^{2k} \alpha_i\binom{n-1}{i}$$
From equation (\ref{codimcentralgk}) 
$c^{gr,z}_n(\mathcal{G}_{2k})=
\ds\binom{n-1}{0}+ \binom{n-1}{2}+\binom{n-1}{4}+\dots+\binom{n-1}{2k-2}$, we obtain
$\alpha_0=\alpha_2=\dots=\alpha_{2k-2}=1$ and $\alpha_1=\alpha_3=\dots=\alpha_{2k-1}=\alpha_{2k}=0$. Thus, we conclude that
$$\chi_n^{gr,z} (\mathcal{G}_{2k})= \ds\sum_{i=0}^{k-1} \chi_{(n-2i, 1^{2i}),\varnothing}\ \ \mbox{and}
\ \ \chi_{n}(\Delta^{gr} (\mathcal{G}_{2k}))=\chi_{(n-2k,1^{2k}),\varnothing}+ \ds\sum_{i=0}^{k-1} \chi_{(n-2i-1, 1^{2i+1}),\varnothing}.$$
\end{proof}

\section{The variety $\textnormal{var}^{gr}(\mathcal{G}^{gr})$}

Let $\mathcal{G}^{gr}$ be the Grassmann algebra with the canonical grading $(\mathcal{G}_{0},\mathcal{G}_{1})$. In the context of polynomial identities, its $T_2$-ideal, graded codimensions, and graded cocharacter decomposition are already known: we have $\Idg(\mathcal{G}^{gr})=\langle [y_1,y_2],z_1\circ z_2, [y_1,z_1] \rangle_{T_2}$, $c_n^{gr}(\mathcal{G}^{gr})=2^{n}$ and $\chi_n^{gr}(\mathcal{G}^{gr})=\ds\sum_{\substack{ 1\leq i \leq n}} \chi_{(n-i),(1^{i})}$ (see \cite{GiaMiZai}). On the other hand, in the context of central polynomials, only its $T_2$-space has been determined: $C^{gr}(\mathcal{G}^{gr})=\langle y_1, \Idg(\mathcal{G}^{gr})\rangle^{T_2}$ (see \cite[Theorem 4.7]{GuiFidKo}). Our contribution will be to compute the central and proper central graded codimension and to present the decomposition of the $n$th central and proper central graded cocharacters.

First, we are concerned with the minimal subvarieties of $\textnormal{var}^{gr}(\mathcal{G}^{gr})$. Denote by $\mathcal{G}_{t}^{gr}$ the algebra $\mathcal{G}_{t}$ with the grading induced by $\mathcal{G}^{gr}$. By \cite[Corollary 7.2]{LaMatalmostGgr} we have that $A$ is a superalgebra that generates a minimal variety in $\text{var}^{gr}(\mathcal{G}^{gr})$ if and only if $\Idg(A) = \Idg(\mathcal{G}_{t}^{gr})$ , for some  $t \geq 2$. We have the following result

\begin{teo}\label{codgkgr}(\cite[Theorem 7.1]{LaMatalmostGgr} and \cite[Theorem 7.3]{NaSaVi}) Let $t \geq 2$, then

\begin{enumerate}
\item $\Idg(\mathcal{G}_{t}^{gr}) = \langle [y_{1},y_{2}],[y,z],z_{1} \circ z_{2},z_{1} \cdots z_{t+1} \rangle_{T_{2}}$;

\item $c_{n}^{gr}(\mathcal{G}_{t}^{gr}) = \ds\sum_{i=0}^{t} \binom{n}{i}$;

\item $\chi^{gr}_{n}(\mathcal{G}_{t}^{gr}) = \ds\sum_{i=0}^{t} \chi_{(n-i),(1^{i})}$.

\end{enumerate}
    
\end{teo}

Recall that the center of $\mathcal{G}$ is equal to $\mathcal{G}_{0}$. On the other hand, the center of the superalgebra $\mathcal{G}_{t}$ is either $\mathcal{G}_{0}$ or $\mathcal{G}_{0}\cup \spn\{e_1 e_2 \cdots e_t\}$, depending on whether $t$ is even or odd, respectively. Therefore, in the study of the $T_{2}$-space, the central graded codimensions and the central graded cocharacters of $\mathcal{G}_{t}^{gr}$, we must take into account the parity of $t$.

\begin{teo} Let $t \geq 2$,
$$C^{gr}(\mathcal{G}_t^{gr})=\langle y_1, z_1z_2\cdots z_t, \id^{gr}(\mathcal{G}_t^{gr})\rangle^{T_2}.$$

\end{teo}

\begin{proof}
     Let $V = \langle y_{1},z_{1}z_{2}\cdots z_{t}, \Idg(\mathcal{G}_{t}^{gr}) \rangle^{T_{2}}$. Notice that $V \subseteq C^{gr}(\mathcal{G}_t^{gr})$. We now prove the reverse inclusion. Let $f \in P_{n} \cap C^{gr}(\mathcal{G}_{t}^{gr})$. By reducing $f$ modulo $\Idg(\mathcal{G}_{t}^{gr})$, there exists $\alpha \in F$ such that $$f \equiv \alpha y_{i_{1}}y_{i_{2}}\cdots y_{i_{r}}z_{j_{1}}z_{j_{2}}\cdots z_{j_{s}}  \mod \Idg(\mathcal{G}_{t}^{gr}).$$

    We analyze the following cases:

    \begin{enumerate}
\item If $s\geq t+1$, so it follows that  $f\equiv 0 \ \mod  \ \Idg(\mathcal{G}_{t}^{gr})$ and therefore $f\in \Idg(\mathcal{G}_{t}^{gr}) \subseteq V$.

\item If $s=t$, since $z_1z_2\cdots z_t \in V$, it follows that
$\alpha \underbrace{y_{i_1}\cdots y_{i_r}z_{j_1}}_{\text{ is odd }}z_{j_2}\cdots z_{j_t} \in V$ and  therefore $f\in V$.

\item If $s< t$, we have to consider the parity of $s$. If $s$ is even, 
$\alpha y_{i_1}\cdots y_{i_r}
\underbrace{z_{j_1}z_{j_2}}_{\text{ is even } }
\cdots
\underbrace{ z_{j_{s-1}} z_{j_s}}_{\text{ is even}} \in V$, then we also conclude that $f\in V$.
  If $s$ is odd, we take the evaluation
$\phi:\phi(y_{i_1})=1,   \dots ,\,  \phi(y_{i_r})=1$,
$\phi(z_{j_1})=e_1, \dots ,\,  \phi(z_{j_s})=e_s.$ Then $\phi(f)=\alpha e_1\cdots e_s \in Z(\mathcal{G}_{t})$. Then $\alpha=0$ and $f\in V$.
\end{enumerate}
Consequently, $C^{gr}(\mathcal{G}_t^{gr})=\langle y_1, z_1z_2\cdots z_t, \id^{gr}(\mathcal{G}_t^{gr})\rangle^{T_2}.$
\end{proof}

\begin{obs} The converse of Remark \ref{rie} is not true. For example, the superalgebras $\mathcal{G}_{2k}^{gr}$ and $\mathcal{G}_{2k+1}^{gr}$ do not satisfy the same polynomial identities however, they do share the same central polynomials. In other words, $C^{gr}(\mathcal{G}_{2k}^{gr}) = C^{gr}(\mathcal{G}_{2k+1}^{gr})$ but $\Id^{gr}(\mathcal{G}_{2k}^{gr}) \neq \Id^{gr}(\mathcal{G}_{2k+1}^{gr})$, for all $k\geq 1$.
    
\end{obs}

\begin{teo}\label{gkgr} Let $k \geq 1$, we have the following.

\begin{enumerate}

\item $\ds c^{gr,z}_n(\mathcal{G}_{2k}^{gr})= \ds c^{gr,z}_n(\mathcal{G}_{2k+1}^{gr})=\ds \sum_{\substack{ 1\leq i \leq 2k \\ i \text{ odd} } } \binom{n}{i}$;

\item $\ds \delta^{gr}_n(\mathcal{G}_{2k}^{gr})=\ds \sum_{\substack{ 0\leq i \leq 2k \\ i \text{  even} } } \binom{n}{i}$ and $\ds \delta^{gr}_n(\mathcal{G}_{2k+1}^{gr}) = \ds \binom{n}{2k+1} +\ds \delta^{gr}_n(\mathcal{G}_{2k}^{gr})$;

\item $\chi^{gr,z}_{n}(\mathcal{G}_{2k}^{gr})= \chi^{gr,z}_{n}(\mathcal{G}_{2k+1}^{gr})=
\ds\sum_{\substack{ 1\leq i \leq 2k \\ i \text{ odd } }} \chi_{(n-i),(1^{i})}.$

\item $\chi_{n}(\Delta^{gr}(\mathcal{G}_{2k}^{gr})) = \ds\sum_{\substack{ 0\leq i \leq 2k \\ i \text{ even } }} \chi_{(n-i),(1^{i})}$ and $\chi_{n}(\Delta^{gr}(\mathcal{G}_{2k+1}^{gr})) =  \chi_{(n-2k-1),(1^{2k+1})}+ \ds \sum_{\substack{ 0\leq i \leq 2k \\ i \text{ even } }} \chi_{(n-i),(1^{i})}$.
\end{enumerate}
\end{teo}
\begin{proof} Let $f \in P_{n-r,r}$. We have $f \equiv \alpha y_{1} \cdots y_{n-r}z_{n-r+1}\cdots z_{n} \mod C^{gr}(\mathcal{G}_{t}^{gr})$, for some $\alpha \in F$. If $ r \geq t$, then $f \equiv 0 \mod C^{gr}(\mathcal{G}_{t}^{gr})$ and we have $c_{n-r,r}^{gr}(\mathcal{G}_{t}^{gr}) = 0$. Now assume $r < t$.
\begin{itemize}
    \item If $r$ is even, then the monomial $y_{1}\cdots y_{n-r}z_{n-r+1}\cdots z_{n}$ is central, so $f \in C^{gr}(\mathcal{G}_{t}^{gr})$ and $c^{gr,z}_{n-r,r}(\mathcal{G}_{t}^{gr})=0$,
    \item If $r$ is odd, then the monomial $y_{1}\cdots y_{n-r}z_{n-r+1}\cdots z_{n}$ is not central, then $c^{gr,z}_{n-r,r}(\mathcal{G}_t^{gr})=1$.
\end{itemize}


Since $c_{n}^{gr,z}(\mathcal{G}_t^{gr}) = \ds\sum_{r=0}^{n} \binom{n}{r} c^{gr,z}_{n-r,r}(\mathcal{G}_{t}^{gr})$, by equation (\ref{codimenr}), we have for $k\ge 1$
$$\ds c^{gr,z}_n(\mathcal{G}_{2k}^{gr})= \ds c^{gr,z}_n(\mathcal{G}_{2k+1}^{gr})=\ds \sum_{\substack{ 1\leq i \leq 2k \\ i \text{ odd} } } \binom{n}{i}.$$

Consequently, using Theorem \ref{codgkgr} and the equation (\ref{eqcodim}), the proper central codimensions is
$$\ds \delta^{gr}_n(\mathcal{G}_{2k}^{gr})=\ds \sum_{\substack{ 01\leq i \leq 2k \\ i \text{  even} } } \binom{n}{i} \  \ \mbox{and} \ \  \ds \delta^{gr}_n(\mathcal{G}_{2k+1}^{gr}) = \ds \binom{n}{2k+1} +\ds \delta^{gr}_n(\mathcal{G}_{2k}^{gr}).$$

Now we determine the decomposition of the $n$th central graded and the $n$th  proper central graded cocha\-racter. From Theorem \ref{codgkgr}, $\chi_{n}^{gr}(\mathcal{G}_{t}^{gr}) = \ds\sum_{i=0}^{t} \chi_{(n-i),(1^i)}$ is the decomposition of $n$th graded cocharacter of $\mathcal{G}_{t}^{gr}$. By equation (\ref{eqcocar}), we have $\chi_{n}^{gr,z}(\mathcal{G}_{t}^{gr}) = \ds\sum_{i=0}^{t} \alpha_i\chi_{(n-i),(1^{i})},$ where  $\alpha_0, \alpha_1, \dots, \alpha_{t}\in \{0,1\}$. We have just verified that for $k \geq 1$, $c_n^z (\mathcal{G}_{2k}^{gr})= c_n^z (\mathcal{G}_{2k+1}^{gr})$. We now proceed to decompose the $n$th central graded cocharacter in the even case, observing that the remaining case can be treated analogously. We have that $c_n^z (\mathcal{G}_{2k}^{gr})= \chi_n^z (\mathcal{G}_{2k}^{gr})(1)$ and so
$$
c_n^z (\mathcal{G}_{2k}^{gr})= \ds\sum_{i=0}^{2k} \alpha_id_{(n-i)(1^{i})} = \ds\sum_{i=0}^{2k} \alpha_i\binom{n}{i}.$$
Since
$c^z_n(\mathcal{G}_{2k}^{gr})=
\ds\binom{n}{1}+ \binom{n}{3}+\dots+\binom{n}{2k-1}$, we obtain
$\alpha_1=\alpha_3=\dots=\alpha_{2k-1}=1$ and $\alpha_0=\alpha_2=\dots=\alpha_{2k}=0$. Thus, we conclude that
$$\chi^{gr,z}_{n}(\mathcal{G}_{2k}^{gr})= \chi^{gr,z}_{n}(\mathcal{G}_{2k+1}^{gr})=
\ds\sum_{\substack{ 1\leq i \leq 2k \\ i \text{ odd } }} \chi_{(n-i),(1^{i})}.$$
As a consequence, using equation (\ref{eqcocar}) and Theorem \ref{codgkgr}, we conclude for $k\geq 1$
$$\chi_{n} (\Delta^{gr}(\mathcal{G}_{2k}^{gr})) = \ds\sum_{\substack{ 0\leq i \leq 2k \\ i \text{ even } }} \chi_{(n-i),(1^{i})} \ \  \mbox{and} \ \  \chi_{n} (\Delta^{gr}(\mathcal{G}_{2k+1}^{gr})) =  \chi_{(n-2k-1),(1^{2k+1})}+ \ds \sum_{\substack{ 0\leq i \leq 2k \\ i \text{ even } }} \chi_{(n-i),(1^{i})}.$$

\end{proof}

Now we have all the necessary ingredients to establish the decomposition of the $n$th graded central cocharacter of the infinite dimensional Grassmann algebra endowed with the canonical grading.

\begin{teo} For the superalgebra $\mathcal{G}^{gr}$, we have:
\begin{enumerate}
 \item $\ds c^{gr,z}_n(\mathcal{G}^{gr})= \ds \sum_{\substack{ 1 \leq i \leq n \\ i \text{ odd} } } \binom{n}{i} = 2^{n-1}$ and $\ds \delta^{gr}_n(\mathcal{G}^{gr})=  \ds \sum_{\substack{ 0 \leq i \leq n \\ i \text{ even} } } \binom{n}{i}= 2^{n-1}$;

\item $\chi^{gr,z}_{n}(\mathcal{G}^{gr})= \ds\sum_{\substack{ 1\leq i \leq n \\ i \text{ odd } }} \chi_{(n-i),(1^{i})}$ and  $\chi_{n} (\Delta^{gr}(\mathcal{G}^{gr})) = \ds\sum_{\substack{ 0\leq i \leq n \\ i \text{ even } }} \chi_{(n-i),(1^{i})}$
\end{enumerate}
    
\end{teo}

\begin{proof}
Consider $f\in P_{n-r,r}^{gr}(\mathcal{G}^{gr})$, so $f$ is generated by $y_1y_2\cdots y_{n-r}z_{n-r+1}z_{n-r+2}\cdots z_{n}$ mod $\id^{gr}(\mathcal{G}^{gr})$. Since $y\in C^{gr}(\mathcal{G}^{gr})$, we have that either $f$ is generated by  $y_1y_2\cdots y_{n-r}z_{n-r+1}z_{n-r+2}\cdots z_{n}$ mod $C^{gr}(\mathcal{G}^{gr})$, if $r$ is odd, or $f\in C^{gr}(\mathcal{G}^{gr})$, if $r$ is even. Then if $r$ is odd, we get $c_{n-r,r}^{gr,z}(\mathcal{G}^{gr}) = 1$, and if $r$ is even, we have $c_{n-r,r}^{gr,z}(\mathcal{G}^{gr}) = 0$. Hence $\ds c^{gr,z}_n(\mathcal{G}^{gr})= \ds \sum_{\substack{ 1 \leq i \leq n \\ i \text{ odd} } } \binom{n}{i}$ and, consequently, $\ds \delta^{gr}_n(\mathcal{G}^{gr})=  \ds \sum_{\substack{ 0 \leq i \leq n \\ i \text{ even} } } \binom{n}{i}$.

Our next step is to analyze the decomposition of the $n$th graded central cocharacter of $\mathcal{G}^{gr}$. Since $\chi^{gr}_n (\mathcal{G}^{gr}) = \ds\sum_{i=0}^{n}\chi_{(n-i),(1^{i})}$ 
we can establish that $\chi^{gr,z}_n (\mathcal{G}^{gr}) = \ds\sum_{i=0}^{n}\alpha_i\chi_{(n-i),(1^{i})}$, $\alpha_{i} \in \{0,1\}$, by comparing the graded cocharacter to the central graded cocharacter.  
Also we have $\chi^{gr,z}_n (\mathcal{G}^{gr}_{2k}) \leq \chi^{gr,z}_n (\mathcal{G}^{gr})$ for all $k\geq 0$ and  by Theorem \ref{gkgr}, we get $\chi^{gr,z}_n (\mathcal{G}^{gr}_{2k})= \ds\sum_{\substack{ 1\leq i \leq 2k \\ i \text{ odd} }} \chi_{(n-i),(1^{i})}$.
So we obtain $\alpha_{i} = 1$, for all $i$ odd and $1 \leq i \leq n$, and  we can write
$$\chi^{gr,z}_n (\mathcal{G}^{gr})=
\ds\sum_{\substack{ 1\leq i \leq n \\ i \text{ odd} }}  \chi_{(n-i),(1^i)}+
\ds\sum_{\substack{ 0\leq j \leq n \\ j \text{ even} }} \alpha_j \chi_{(n-j),(1^j)}, \ \textnormal{with} \ \alpha_{j} \in \{0,1\}.$$ 

Since $c_n^{gr,z} (\mathcal{G}^{gr})= \chi_n^{gr,z} (\mathcal{G}^{gr})(1) = 2^{n-1}$ and $\ds\sum_{\substack{ 1\leq i \leq n \\ i \text{ odd} }}  d_{(n-i),(i)}=2^{n-1}$, we obtain $\alpha_{j} = 0$ for all $j$ even and $0 \leq j \leq n$. Hence, we conclude that $\chi^{gr,z}_n (\mathcal{G}^{gr})=
\ds\sum_{\substack{ 1\leq i \leq n \\ i \text{ odd} }} \chi_{(n-i),(1^{i})}$ and $\chi_{n} (\Delta^{gr}(\mathcal{G}^{gr})) = \ds\sum_{\substack{ 0\leq i \leq n \\ i \text{ even } }} \chi_{(n-i),(1^{i})}.$ 

\end{proof}

\end{document}